\documentclass[11pt]{article}
\usepackage{caption}
\usepackage[a4paper,margin=1in]{geometry}
\usepackage{amsmath,amssymb,amsthm}
\usepackage{mathtools}
\usepackage{tikz}
\usetikzlibrary{arrows.meta,calc,positioning}
\usepackage{authblk}
\usepackage{cite}

\newtheorem{theorem}{Theorem}
\newtheorem{problem}{Problem}
\newtheorem{definition}{Definition}
\newtheorem{claim}{Claim}

\newtheorem{fact}{Fact}

\newtheorem{lemma}{Lemma}

\newtheorem{proposition}{Proposition}

\title{Characterizing forbidden induced subgraphs that force top
vertices to be Gallai vertices}
\author{Yurui Tang \thanks{Email address: tyr2290@163.com }}
\affil{Department of Mathematics, East China Normal University, Shanghai, 200241, China}
\date{}

\begin{document}
\maketitle

\begin{abstract}
A vertex of a graph is  called a top vertex if it has maximum degree in the graph. A vertex of a graph is called a Gallai vertex if it belongs to every longest path of the graph.  Golan and Shan  proved that every top vertex of any connected induced-$2P_2$ free graph is a Gallai vertex. Long, Milans, and Munaro  subsequently showed that if every connected induced-$H$ free graph has a Gallai vertex, then $H$ must be a linear forest of order at most nine. They also proved that, for every linear forest $H$ of order  at most four, every top vertex of any connected induced-$H$ free graph is a Gallai vertex.

In this paper, we determine the graphs $H$ for which every top vertex of any connected induced-$H$ free graph is a Gallai vertex. Our result shows that this holds precisely when $H$ is a linear forest of order at most four or $H=P_3+2P_1$.
\end{abstract}

{\bf Keywords.} Gallai vertex; Top vertex; Longest path; Induced-$H$ free

{\bf Mathematics Subject Classification.} 05C38, 05C75
\vskip 8mm
\section{Introduction}
We consider finite simple graphs and use standard terminology and
notation from \cite{bondy}. Let $G$ be a graph with vertex set $V(G)$
and edge set $E(G)$. The order of $G$ is denoted by $|G|:=|V(G)|.$
For a vertex $v\in V(G)$, let $N_G(v)$ and $deg_G(v)$ denote the
neighbourhood and degree of $v$ in $G$, respectively. For $S\subseteq V(G)$, set $N_G(S)=\{u\in V(G)\setminus S:u~ \text{has a neighbour in}~ S\}$.
For a subgraph $H$ of $G$, set $N_G(H)=N_G(V(H))$.  The maximum
degree of $G$ is $\Delta(G):=\max\{deg_G(v):v\in V(G)\}.$
For a set $S\subseteq V(G)$, let $G[S]$ denote the subgraph of $G$
induced by $S$, and write
$G-S:=G[V(G)\setminus S].$  For graphs we will use equality up to isomorphism, so $G = H$ means that $G$ and $H$ are isomorphic. 

Denote $P_n$ and $K_n$ by the path and the complete graph of order
$n$, respectively, and $K_{m,n}$ for the complete bipartite graph
with parts of orders $m$ and $n$. For two vertex-disjoint graphs $G$
and $H$, their disjoint union is denoted by $G+H$. For a positive
integer $k$, the graph $kG$ denotes the disjoint union of $k$ copies
of $G$. A \emph{linear forest} is a graph each of whose
components is a path.
A longest path of a graph is a path with the maximum possible number
of vertices. 

A vertex of $G$ is called a \emph{Gallai vertex} if it
belongs to every longest path of $G$. Following \cite{HZ}, we call a vertex of degree
$\Delta(G)$ a \emph{top vertex}. For a fixed graph $H$, a
graph $G$ is called \emph{induced-$H$ free} if $G$ contains no
induced subgraph isomorphic to $H$.

Gallai \cite{Gallai} asked whether every connected
graph contains at least one Gallai vertex. This was answered in the negative \cite{W,WV,Z2}. 
After that, many researchers began to investigate which graphs contain a Gallai vertex. 
It is well known that every tree contains a Gallai vertex. Other such graph classes have also  been identified (see \cite{CL,CF,C,CEF,Gs,J,LV,LM,N,RS,SZ}).

In 2018, Golan and Shan \cite{GS} proved that every top vertex of a connected
induced-$2P_2$ free graph is a Gallai vertex.
Their result motivates
the following general problem.

\begin{problem}
For which  graph  $H$ is it true that every top vertex of any connected induced-$H$ free graph is a Gallai vertex?
\end{problem}

Recently, Long, Milans, and Munaro \cite{LMM} established strong necessary
conditions for the existence of Gallai vertices in hereditary graph
classes. In particular, they proved the following result.
\begin{theorem}[Long--Milans--Munaro~\cite{LMM}]\label{lmm1}
If every connected induced-$H$ free graph has a Gallai vertex, then $H$
must be a linear forest of order at most nine. 
\end{theorem}
They also proved the following partial result of Problem 1.
\begin{theorem}[Long--Milans--Munaro~\cite{LMM}]\label{lmm2}
Let $H$ be a linear forest of order at most four. Then each top vertex of every connected induced-$H$ free graph is a Gallai vertex.
\end{theorem}
Our main result answers Problem 1.
\begin{theorem}\label{main}
Let $H$ be a graph. Then every top vertex of any connected induced-$H$ free graph is a Gallai vertex if and only if $H$ is a linear forest of order at most four or $H = P_3 + 2P_1$.
\end{theorem}

The remainder of the paper is organized as follows. In Section 2, we
introduce the notation and collect several basic properties of longest
paths. Section 3 is divided into several parts. We first show two Propositions. In Proposition~\ref{P_3+2P_1}, we prove that every top vertex of every connected induced-$(P_3+2P_1)$ free graph is a
Gallai vertex. In Proposition~\ref{P_2+3P_1}, we prove that every top vertex of every connected induced-$(P_2+3P_1)$ free graph is a
Gallai vertex except for a family of graphs.  Next, we construct counterexamples for the remaining linear forests of order five. In addition, we use a Lemma to handle all linear forests of order at least six and finish the proof of Theorem \ref{main}. 
\section{Preliminaries}
Let $G$ be a graph.  	For two disjoint vertex sets $A,B\subseteq V(G)$, we say that $A$ is
\emph{complete} to $B$ if every vertex of $A$ is adjacent to every
vertex of $B$, and that $A$ is \emph{anticomplete} to $B$ if no edge
joins a vertex of $A$ to a vertex of $B$. A component consisting of a
single vertex is called a \emph{trivial component}.
For two distinct vertices $x$ and $y$ in $G$, an {\em $(x, y)$-path}  is a path whose endpoints are $x$ and $y$.  An {\em $x$-path} is a path with endpoint $x$.
For two sets $U,W\subseteq V(G)$, we write
$U\leftrightarrow W$ if $U$ is complete to $W$, and
$U\nleftrightarrow W$ if $U$ is anticomplete to $W$.
When $U=\{u\}$ and $W=\{w\}$, we simply write $u\leftrightarrow w$ and
$u\nleftrightarrow w$, respectively.
A set of vertices  is {\em independent} if no two of its 
elements are adjacent. We call $G$  \emph{traceable}(resp. \emph{hamiltonian}) if $G$ contains a Hamilton path(resp. cycle). Otherwise, $G$ is \emph{nontraceable}(resp. \emph{nonhamiltonian}).

Let $P=v_1v_2\dots v_l$ be a path in a graph $G$.
We set 
\begin{align*}
\text{
$v_i^+:=v_{i+1}$,~~ 
$v_j^-:=v_{j-1}$,~~  
$v_i^{+p}:=v_{i+p}$~~and~~
$v_j^{-q}:=v_{j-q}$,
}\end{align*}
respectively.  
For any vertex subset $S\subseteq V(P)$, 
we write 
\begin{align*}
S^{+p}:=\{v_i^{+p} : v_i\in S\setminus \{v_{l-p+1},\ldots,v_l\}\}\text{~~and~~}
S^{-q}:=\{v_i^{-q} : v_i\in S\setminus \{v_1,\ldots,v_{q}\}\}.
\end{align*}  
Set $S^{+}=S^{+1}$ and $S^{-}=S^{-1}$ for short. 
If $a<b,$
we set 
\begin{align*}
v_a \overrightarrow{P} v_b:=v_a v_{a+1}\ldots  v_b \text{~~and~~} 
v_b \overleftarrow{P} v_a:=v_b v_{b-1}\ldots  v_a.
\end{align*}
\begin{lemma}[Moon--Moser~\cite{MM}]\label{MM}
 Let $G$ be a balanced bipartite graph with bipartition $(X,Y)$, where $|X|=|Y|=n\ge 2$. If $deg_G(x)+deg_G(y)> n$ for every nonadjacent pair $x\in X$ and $y\in Y$, then $G$ is hamiltonian.
\end{lemma}

\begin{lemma}[Golan--Shan~\cite{GS}]\label{pre}
Let $G$ be a connected nontraceable graph.
Let $P=v_1v_2\cdots v_\ell$ be a longest path  of $G$. Let $H$ be a  component of $G-V(P)$. Set $R=N_G(H)\cap V(P)$.
Then the following statements hold.
\begin{enumerate}
\item[\rm(1)] $G[V(P)]$ is nonhamiltonian. In particular, if $v_1\leftrightarrow v_i$, then  $v_\ell\nleftrightarrow v_{i-1}.$
\item[\rm(2)] $v_1,v_\ell \notin R.$
\item[\rm(3)] If $v_i\in R$, then $v_{i-1},v_{i+1}\notin R$.
\item[\rm(4)] If $v_i\in R$, then $v_1\nleftrightarrow v_{i+1}$ and $v_\ell\nleftrightarrow v_{i-1}$.
\item[\rm(5)] If $v_i,v_j\in R$, where $i\ne j$, then
$v_{i-1}\nleftrightarrow v_{j-1}$ and
$v_{i+1}\nleftrightarrow v_{j+1}$.
\end{enumerate}
\end{lemma}
\begin{proof}[Proof of Lemma \ref{pre}]
We first prove (1). Suppose to the contrary that $G[V(P)]$ is hamiltonian. Choose a vertex $v_i\in R$ with $v_i\leftrightarrow u$, where $u\in V(H)$. Since $G[V(P)]$ is hamiltonian, let $P'$ be a $v_i$-path that contains all vertices of $V(P)$. Hence, $uv_iP'$ is a longer path than $P$, which contradicts the maximality of $P$.

The proofs of (2)-(5) are given in \cite{GS} and are omitted here.
\end{proof}

\section{Proofs}
\begin{proposition}\label{P_3+2P_1}
Every top vertex of every connected induced-$(P_3+2P_1)$ free graph is a
Gallai vertex.
\end{proposition}
\begin{proof}[Proof of Proposition~\ref{P_3+2P_1}]
If $G$ contains a Hamilton path, then every top vertex is a Gallai vertex.
Next we just consider a nontraceable graph $G$. Suppose to the contrary that $G$ has a top vertex $u$ that is not a
Gallai vertex. Then there exists a longest path of $G$ that does not
contain $u$. Let
$$
P=v_1v_2\cdots v_s
$$
be such a longest path. Let $H$ be the component of $G-V(P)$ containing
$u$, and set
$$
R=N_G(H)\cap V(P).
$$
Since $G$ is connected, we have
$R\ne\emptyset.$
We first record the following fact.
\begin{fact}\label{fact1}
The sets $R^+\cup\{v_1\}$ and $R^-\cup\{v_s\}$ are independent.
\end{fact}
\begin{proof}[Proof of Fact \ref{fact1}]
We just prove that $R^+\cup\{v_1\}$ is independent. The proof for $R^-\cup\{v_s\}$ is symmetric.

First, by Lemma \ref{pre}(2), $v_1\notin R$. If $v_i\in R$ and $v_1\leftrightarrow v_i^+$, then, since $v_i$ has a neighbour in $H$, say $w$, the path
$$
wv_i \overleftarrow{P} v_1 v_i^+ \overrightarrow{P} v_s
$$
is longer than $P$, a contradiction. Hence $v_1$ is nonadjacent to every vertex of $R^+$.
By Lemma \ref{pre}(5), $R^+$ is an independent set.
Thus $R^+\cup\{v_1\}$ is independent.
\end{proof}

We first determine the structure of $H$.
\begin{claim}\label{Hcomp}
The subgraph $H$ is complete.
\end{claim}

\begin{proof}[Proof of Claim \ref{Hcomp}]
Suppose not. Since $H$ is connected, $H$ contains an induced path $xyz$ of order $3$. By Lemma \ref{pre}(1) and (2), the vertices $v_1$ and $v_s$ are nonadjacent to every vertex of $H$, and $v_1\nleftrightarrow v_s$. Hence, the set
$$\{x,y,z\}\cup\{v_1,v_s\}$$ induces a copy $P_3+2P_1$, a contradiction. Therefore $H$ is complete.
\end{proof}

We distinguish two cases.

\medskip
\noindent\textbf{Case 1.} There exists a vertex $v_i\in R$ that is complete to $V(H)$.

By Lemma \ref{pre}(3), $v_{i+1}\notin R$. Then
$v_{i+1}v_i u$
is an induced path.
By Fact~\ref{fact1}, 
$R^+\cup\{v_1\}$ is an independent set.

We assert that $v_i$ is nonadjacent to at most one vertex of
$(R^+\setminus\{v_{i+1}\})\cup\{v_1\}$. Otherwise, choose two distinct vertices
$w_1,w_2\in (R^+\setminus\{v_{i+1}\})\cup\{v_1\}$
such that
$v_i\nleftrightarrow \{w_1,w_2\}$.
Note that
$u\nleftrightarrow \{w_1, w_2\}.$
Moreover, by Fact~\ref{fact1}, $w_1\nleftrightarrow w_2$ and $v_{i+1}\nleftrightarrow w_j$ for $j=1,2$. Hence
$$
\{v_{i+1},v_i,u\}\cup\{w_1,w_2\}
$$
induces a copy of $P_3+2P_1$, a contradiction.

Since $v_i$ is also adjacent to $v_{i+1}$ and to every vertex of $H$, we obtain
$$
\begin{aligned}
deg_G(v_i)
&\ge (\left|(R^+\setminus\{v_{i+1}\})\cup\{v_1\}\right|-1)
+1+|V(H)|        \\
&= |R|+|V(H)|.
\end{aligned}
$$
On the other hand, 
$deg_G(u)\le |V(H)|-1+|R|.$
Thus,
$deg_G(v_i)>deg_G(u)=\Delta(G),$
a contradiction.

\medskip
\noindent\textbf{Case 2.} No vertex of $R$ is complete to $V(H)$.
Set $t=|H|$. Then $t\ge 2$.
For every $v_i\in R$, there exists a vertex $y_i\in V(H)$ such that
$v_i\nleftrightarrow y_i.$
Let
$$
q=\min\{j:v_j\in R\}.
$$
Choose $x,y\in V(H)$ such that
$v_q\leftrightarrow x$ and $ v_q\nleftrightarrow y.$
Since $H$ is complete, $x\leftrightarrow y$, and hence
$v_qxy$
is an induced path of order $3$.

By Lemma \ref{pre}(1) and (2), both $v_1$ and $v_s$ are nonadjacent to every vertex of $H$, and $v_1\nleftrightarrow v_s$. Consider that the induced path
$v_qxy$ and the vertices $v_1$ and $v_s$. Since $G$ is induced-$(P_3+2P_1)$ free,  we obtain
$$
v_1\leftrightarrow v_q
\quad\text{or}\quad
v_s\leftrightarrow v_q.
$$

\begin{claim}\label{both}
$v_1\leftrightarrow v_q$ and $v_s\leftrightarrow v_q.$
\end{claim}
\begin{proof}[Proof of Claim \ref{both}]
Suppose to the contrary that exactly one of $v_1$ and $v_s$ is
adjacent to $v_q$.
We first show that $V(v_{q+1}\overrightarrow{P}v_{s-1})\subseteq N_G(v_s).$ 

Suppose $v_1\leftrightarrow v_q$ and  $v_s\nleftrightarrow v_q$. By Fact \ref{fact1}, $R^+\cup \{v_1\}$ is independent. Then $
v_1v_qv_{q+1}$ is an induced path. 
Consider the path $v_1v_qv_{q+1}$ and two vertices $y,v_s$.
Since $y\nleftrightarrow \{v_q, v_{q+1}\}$, the induced-$(P_3+2P_1)$ freeness of $G$ implies that 
$v_s\leftrightarrow v_{q+1}.$ Set
$$
h=\max\left\{i: V(v_{q+1}\overrightarrow{P} v_i)\subseteq N_G(v_s)\right\}.
$$ 
Then $h\ge q+1$.  It suffices to show that $h=s-1$. Suppose to the contrary that $h<s-1$. Then $v_h\leftrightarrow v_s$ and $v_{h+1}\nleftrightarrow v_s.$
Thus $v_sv_hv_{h+1}$ is an induced path.
We assert that $v_1\nleftrightarrow \{v_h,v_{h+1}\}$. Since $v_s\leftrightarrow v_h,$ by Lemma \ref{pre}(1), $v_1\nleftrightarrow v_{h+1}$.
If $h=q+1$, by Fact \ref{fact1}, $v_1\nleftrightarrow v_h$. If $h\ge q+2$, since $v_s\leftrightarrow v_{h-1},$ by Lemma \ref{pre}(1), $v_1\nleftrightarrow v_{h}$.
By Fact \ref{fact1}, $R^-\cup\{v_s\}$ is an independent set.
It follows that $y\nleftrightarrow \{v_h,v_{h+1}\}$.
Consequently, $$\{v_s,v_h,v_{h+1}\}\cup\{v_1,y\}$$ induces a copy of $P_3+2P_1$, a contradiction. Therefore
$h=s-1.$

Suppose  $v_1\nleftrightarrow v_q$ and $v_s\leftrightarrow v_q$. Set
$$
h'=\max\left\{i: V(v_{q}\overrightarrow{P} v_i)\subseteq N_G(v_s)\right\}.
$$ Thus $h'\ge q$. It suffices to show that $h'=s-1$.
By Lemma \ref{pre}(1), $v_1\nleftrightarrow V(v_{q+1}\overrightarrow{P} v_{h'+1}).$ 
Suppose to the contrary that $h'<s-1$. Then
$v_{h'}\leftrightarrow v_s$ and $v_{h'+1}\nleftrightarrow v_s,$
and hence $v_sv_{h'}v_{h'+1}$ is an induced path. If $h'\ge q+1$, by Fact \ref{fact1},  $y\nleftrightarrow \{v_{h'},v_{h'+1},v_s\}$ since  $v_s\leftrightarrow \{v_{h'-1},v_{h'}\}$. If $h'=q$, $y\nleftrightarrow \{v_{h'},v_{h'+1},v_s\}$.
Since $v_1\nleftrightarrow \{v_s,v_{h'},v_{h'+1}\}$, 
$$\{v_s,v_{h'},v_{h'+1}\}\cup\{v_1,y\}$$ induces a copy of $P_3+2P_1$, a contradiction. 
Therefore
$h'=s-1.$ 

Thus, $V(v_{q+1}\overrightarrow{P}v_{s-1})\subseteq N_G(v_s).$ By Fact \ref{fact1}, $R^-\cup\{v_s\}$ is independent, which implies  $R^- \cap V(v_{q+1}\overrightarrow{P}v_s)=\emptyset$. By the definition of $q$, combine with $v_{q+1}\notin R$, it follows that $R=\{v_q\}$. Since $u$ is a top vertex of $G$, we have that $u\leftrightarrow v_q$. Hence $deg_G(u)=t.$  By Claim \ref{Hcomp}, $H$ is a complete graph. Let $Q$ be a $u$-path in $H$ that contains all vertices of $H$. By the maximality of $P$, we have that $$|V(v_{1}\overrightarrow{P} v_{q-1})|\ge t \qquad\text{and}\qquad |V(v_{q+1}\overrightarrow{P} v_s)|\ge t.$$
By Lemma \ref{pre}(1), we have that $v_1\nleftrightarrow V(v_{q+1}\overrightarrow{P}v_s).$ 
Set
$$
m=\max\left\{j: V(v_{q+1}\overrightarrow{P} v_j)\subseteq N_G(v_q)\right\}.
$$
Since $deg_G(v_q)\le deg_G(u)=t$ and $|V(v_{q+1}\overrightarrow{P} v_s)|\ge t$, we have $m<s$. Hence
$v_q\nleftrightarrow v_{m+1}.$
Thus
$v_qv_mv_{m+1}$
is an induced path.

Recall that $|V(v_1\overrightarrow{P} v_{q-1})|\ge t$. Let
$z\in V(v_1\overrightarrow{P} v_t).$
We assert that $z\nleftrightarrow V(v_{q+1}\overrightarrow{P} v_s)$. Otherwise, suppose that
$z\leftrightarrow w$
for some $w\in V(v_{q+1}\overrightarrow{P} v_s)$. If $w\in V(v_{q+2}\overrightarrow{P} v_s)$, then		
$$
v_{q-1}\overleftarrow{P} z w\overrightarrow{P} v_s w^-\overleftarrow{P} v_q uQ
$$
is a path longer than $P$, a contradiction. If $w=v_{q+1}$, then 
$$v_{s}\overleftarrow{P}v_{q+1}z\overrightarrow{P}v_{q-1}v_quQ$$
is a path longer than $P$, a contradiction.

Now consider the induced path
$v_qv_mv_{m+1}$
and two vertices $z,y$. Recall that $y\nleftrightarrow \{v_q,v_m,v_{m+1}\}.$
Since $G$ is induced-$(P_3+2P_1)$ free, it follows that
$z\leftrightarrow v_q.$
Since $z$ was chosen arbitrarily, $v_q$ is adjacent to every vertex of $V(v_1\overrightarrow{P} v_t)$. Therefore
$$
deg_G(v_q)\ge |V(v_1\overrightarrow{P} v_t)|+|\{v_{q+1},u\}|=t+2>t=deg_G(u),
$$
which contradicts the assumption that $u$ is a top vertex.

This completes the proof of Claim \ref{both}.
\end{proof}

\begin{claim}\label{R=1}
$|R|=1$.
\end{claim}
\begin{proof}[Proof of Claim \ref{R=1}]
By Claim \ref{both} and the fact that $v_1\nleftrightarrow v_s$, $v_1v_qv_s$ is an induced path. 
If $|R|=1$, we are done, so assume that $|R|\ge 2$. Choose $v_p\in R\setminus\{v_q\}$.  By Lemma \ref{pre}(3), $y\nleftrightarrow \{v_{p-1}, v_{p+1}\}$. By Fact \ref{fact1}, $v_s\nleftrightarrow v_{p-1}$. Consider the path $v_1v_qv_s$ and two vertices $y,v_{p-1}$, by induced-$(P_3+2P_1)$ freeness,
it follows that $v_q\leftrightarrow v_{p-1}$ or $v_1\leftrightarrow v_{p-1}.$ 
Choose a vertex $w\in V(H)$ such that $w\leftrightarrow v_p$.

We assert that $v_q\leftrightarrow v_{p-1}$. Otherwise, suppose $v_q\nleftrightarrow v_{p-1}$. Then $v_1\leftrightarrow v_{p-1}.$ Now $$v_{q+1}\overrightarrow{P}v_{p-1}v_1\overrightarrow{P}v_qv_s\overleftarrow{P}v_pw$$
is a longer path than $P$, a contradiction.
Since $v_p$ was chosen arbitrarily, we have that $R^-\subseteq N_G(v_q)$. Let $$g=\min\{i:v_i\in R\setminus \{v_q\}\}.$$ Note that $v_sv_qv_{g-1}$ is an induced path of order three. Recall that $y\nleftrightarrow \{v_s,v_q,v_{g-1}\}$. 

By the maximality of $P$, $t\le q-1$.  We assert that $v_q\leftrightarrow V(v_1\overrightarrow{P}v_t).$ Otherwise, since $v_1\leftrightarrow v_q$, there exists an
integer $f$ with $2\le f\le t$ such that
$v_f\nleftrightarrow v_q$. Since $G$ is induced-$(P_3+2P_1)$ free, consider the induced path $v_sv_qv_{g-1}$ and vertex set $\{y,v_f\}$, we have that $v_f\leftrightarrow v_{g-1}$ or $v_f\leftrightarrow v_{s}$. 
If $v_f\leftrightarrow v_{g-1}$,  let $Q_1$ be a $w$-path in $H$ that contains all vertices of $H$.  Then $$v_{q-1}\overleftarrow{P}v_fv_{g-1}\overleftarrow{P} v_qv_s\overleftarrow{P}v_gw Q_1$$
is a longer path than $P$, a contradiction.
Thus, $v_f\leftrightarrow v_{s}$. Recall that $x\in V(H)$ and $x\leftrightarrow v_q$. Let $Q_2$ be a $x$-path in $H$ that contains all vertices of $H$.  Then
$$v_{q-1}\overleftarrow{P}v_fv_s\overleftarrow{P}v_qxQ_2$$
is a longer path than $P$, a contradiction.
Therefore, $v_q\leftrightarrow V(v_1\overrightarrow{P}v_t).$

Since $R^-\subseteq N_G(v_q)$ and $v_q\leftrightarrow V(v_1\overrightarrow{P}v_t),$ it follows that
$$deg_G(v_q)\ge |R^-|+|V(v_1\overrightarrow{P}v_{t-1})|+|\{v_{s}\}|=|R|+|H|>deg_G(u),$$
which contradicts that $u$ is top vertex of $G$.
This completes the proof of Claim \ref{R=1}.
\end{proof}

By Claim \ref{R=1}, since $u$ is a top vertex, $u\leftrightarrow v_q.$ Thus, there exists at least one vertex of $V(v_1\overrightarrow{P} v_{q-1})$ that is nonadjacent to $v_q$.  Set $$a=\min \{ i: V(v_{i+1}\overrightarrow{P}v_{q-1})\subseteq N_G(v_q)\}.$$ 
Then $v_a\nleftrightarrow v_q$. Let $r=t-(q-a-1).$	We assert that $V(v_{q+1}\overrightarrow{P}v_{q+r})\subseteq N_G(v_q).$ Otherwise, suppose  $v_b\in V(v_{q+1}\overrightarrow{P}v_{q+r}) $ such that $v_b\nleftrightarrow v_q$. Consider the induced path $v_quy$ and vertices $v_a, v_b$. Since $G$ is induced-$(P_3+2P_1)$ free, $v_a\leftrightarrow v_b$. Let $Q$ be a $u$-path in $H$ that contains all vertices of $H$. Then  we find a path
$$P'=v_{1}\overrightarrow{P}v_av_b\overrightarrow{P}v_sv_quQ.$$
Note that $|V(P)\setminus V(P')|=(q-a-1)+(b-q-1)=t-r+(b-q-1)\le t-1$. This contradicts the maximality of $P$.

Therefore, $$deg_G(v_q)\ge |V(v_{q+1}\overrightarrow{P}v_{q+r})|+|V(v_{a+1}\overrightarrow{P}v_{q-1})|+|\{u\}|=t+1>deg_G(u),$$
a contradiction.
This completes the proof of Proposition~\ref{P_3+2P_1}.
\end{proof}
Theorem \ref{lmm2} together with Proposition \ref{P_3+2P_1} prove the sufficient direction of Theorem \ref{main}. 
To prove the converse direction of Theorem~\ref{main}, we need to analyze the remaining linear forests. 
We first consider linear forest $P_2+3P_1$ and show Proposition \ref{P_2+3P_1}. 

\begin{definition}\label{def:G}
	For each integer $t\ge 3$, let $K=K_{t,t+2}$
	be a complete bipartite graph with bipartition $(X,Y)$, where	$|X|=t$ and $|Y|=t+2.$
	Choose a vertex $u\in Y$, and let $Y_0=Y\setminus\{u\}.$
	
	For $t=3$, let $\mathcal{F}_3$ be the family of all spanning
	subgraphs $F$ of $K[X,Y_0]$ such that $F=P_7.$
	
	For each integer $t\ge 4$, let $\mathcal{F}_t$ be the family of all
	spanning subgraphs $F$ of $K[X,Y_0]$ of the form
	$F=Q+C_1+\cdots+C_s,$
	where $s\ge 0$, the graphs $C_1,\dots,C_s$ are even cycles, and
	$Q$ is a path satisfying one of the following conditions:
	\begin{enumerate}
		\item[\rm(i)] $Q$ is nontrivial and both endpoints of $Q$ belong
		to $Y_0$;
		
		\item[\rm(ii)] $Q$ is trivial and consists of a single vertex of
		$Y_0$.
	\end{enumerate}
	Here the components
	$Q,C_1,\dots,C_s$ are pairwise vertex-disjoint.
	
	For every $F\in\mathcal{F}_t$, define $V(G_F)=X\cup Y$ and $E(G_F)=E(K_{t,t+2})\setminus E(F).$
	Set $\mathcal{G}_t=\{G_F:F\in\mathcal{F}_t\}$ and $\mathcal{G}=\bigcup_{t\ge 3}\mathcal{G}_t.$
\end{definition}
\begin{figure}[h]
	\centering
	\includegraphics[width=0.9\textwidth]{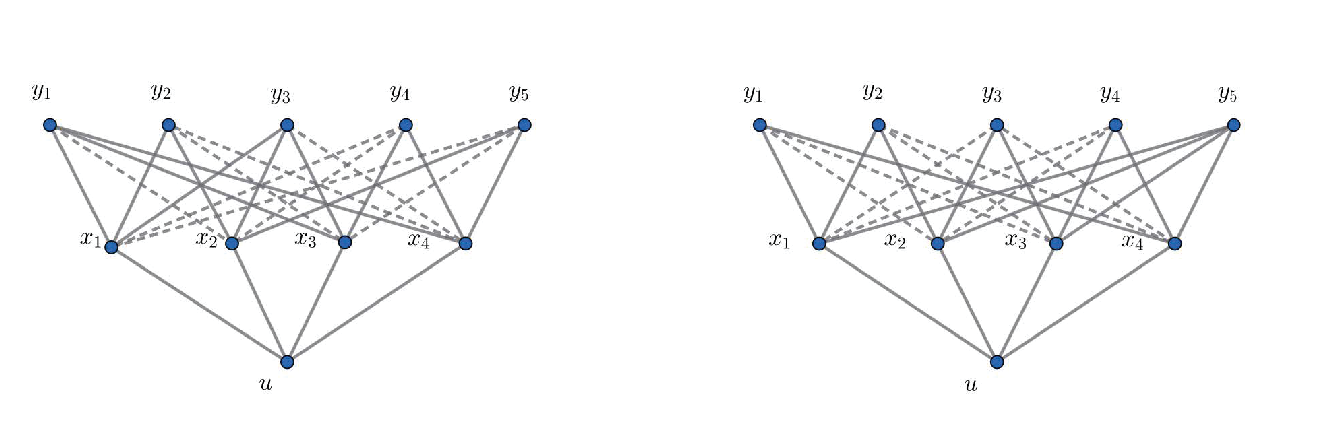}
	\caption{Two graphs in $\mathcal{G}_4$ (The left dashed line: $F=P_9$; the right dashed line: $F=C_8+P_1$). }
\end{figure}
We have the following proposition.

\begin{proposition}\label{P_2+3P_1}
A graph $G$ is a connected induced-$(P_2+3P_1)$ free graph
containing a top vertex that is not a Gallai vertex if and only if
$G\in\mathcal G$.
\end{proposition}

\begin{proof}[Proof of Proposition~\ref{P_2+3P_1}]
We first prove the sufficient condition. Let $G\in\mathcal G$.
Then $G=G_F$ for some $F\in\mathcal F_t$ and some integer $t\ge3$. Clearly, $G$ is connected.
Let $(X,Y)$ be the bipartition in Definition~\ref{def:G}, where
$Y_0=Y\setminus\{u\}$.  
By the definition of $\mathcal F_t$, we have
$\deg_F(x)=2$ for every $x\in X$ and
$\deg_F(y)\le2$ for every $y\in Y_0$. Since $u$ is complete to
$X$ in $G$, it follows that
$$
\deg_G(u)=t\quad \text{and}\quad
\deg_G(x)=t,\quad\text{for every }x\in X,
$$
and
$$
\deg_G(y)=t-\deg_F(y)\le t,
\quad\text{for every }y\in Y_0.
$$
Therefore, $\Delta(G)=t$, and hence $u$ is a top vertex of $G$.

We next show that $u$ is not a Gallai vertex. Since $G$ is bipartite
with partite sets $X$ and $Y$, where $|X|=t$ and $|Y|=t+2$, every
path of $G$ contains at most $t$ vertices of $X$ and at most $t+1$
vertices of $Y$. Thus, every path of $G$ has order at most $2t+1$.
It suffices to prove that $G-u$ contains a Hamilton path. First suppose that
$t=3$. Since $F=P_7$, it follows that there is a Hamilton path of $G-u$.
Now suppose that $t\ge4$. Let $H=G-u$, and construct a balanced
bipartite graph $H^*$ from $H$ by adding a new vertex $z$ to the
part containing $X$ and joining $z$ to every vertex of $Y_0$.
Thus, the two partite sets of $H^*$ are $X\cup\{z\}$ and $Y_0$,
both of order $t+1$.
For any two vertices $x\in X$ and $y\in Y_0$  that $x$ and $y$ are nonadjacent in $H^*$, we have that $\deg_{H^*}(x)=t-1$ and $\deg_{H^*}(y)=t+1-\deg_F(y)\ge t-1.$ Since $t>3$, $2t-2> t+1$.
By Lemma \ref{MM}, $H^*$ contains a Hamilton cycle.
Deleting $z$ from this cycle yields a Hamilton path of $H=G-u$.

It remains to prove that $G$ is induced-$(P_2+3P_1)$ free. Suppose
to the contrary that there exists an edge $xy$ and a set
$W=\{w_1,w_2,w_3\}$ such that they induce a copy of $P_2+3P_1$. We may assume
that $x\in X$ and $y\in Y$. Thus, $W$ is independent and
anticomplete to $\{x,y\}$.

If $y=u$, then $W\subseteq Y_0$, because $u$ is complete to $X$.
Hence $x$ is nonadjacent to three vertices of $Y_0$, a contradiction to
$\deg_F(x)=2$.
Suppose that $y\in Y_0$. Let $p=|W\cap X|$. If
$p=0$, then $\deg_F(x)\ge3$, while if $p=3$, then
$\deg_F(y)\ge3$. Both cases are impossible.
Suppose $p=1$. Let $x'$ be the unique vertex of $W\cap X$. 
$x'$ is nonadjacent to $y$ and the two vertices of $W\cap Y_0$. Hence $\deg_F(x')\ge3$, a contradiction. Thus, $p=2$. Let $y'$ be the
unique vertex of $W\cap Y_0$. $y'$ is nonadjacent to $x$
and the two vertices of $W\cap X$. Hence $\deg_F(y')\ge3$, again a
contradiction.
Therefore, $G$ is induced-$(P_2+3P_1)$ free. Thus, we prove the sufficient condition.

Now we suppose $G$ has a top vertex $u$ that is not a
Gallai vertex. It suffices to prove that  $G\in \mathcal{G}$. Note that $G$ is nontraceable. Then there exists a longest path of $G$ that does not
contain $u$. Let
$$
P=v_1v_2\cdots v_s
$$
be such a longest path. Let $H$ be the component of $G-V(P)$ containing
$u$, and set
$$
R=N_G(H)\cap V(P).
$$
Since $G$ is connected, we have
$R\ne\emptyset.$
Similar to the above proof, we have  the following Fact.
\begin{fact}\label{fact2}
The sets $R^+\cup\{v_1\}$ and $R^-\cup\{v_s\}$ are independent.
\end{fact}
Set
$
k=\min\{i:v_i\in R\},$ and
$\ell=\max\{i:v_i\in R\}.$ 
We now distinguish two cases according to the order of the component
$H$.

\medskip
\noindent\textbf{Case 1.} $|V(H)|\ge 2$.

We first show that $|R|=1$. Suppose to the contrary that $|R|\ge 2$. Choose an edge $xy\in E(H)$. Then the set
$$
\{x,y\}\cup\{v_1\}\cup R^+
$$
contains an induced copy of $P_2+3P_1$, since $R^+\cup\{v_1\}$ is independent and has at least three vertices. This is a contradiction. Therefore
$|R|=1.$ Let $R=\{v_k\}.$
Since $N_G(u)\subseteq V(H)\cup\{v_k\}$ and $deg_G(u)=\Delta(G)\ge 3$, we have
$|V(H)|\ge 3.$

Let $xy\in E(H)$. By the maximality of $P$, it follows that $k\ge 3$ and $k\le s-2$. Considering the edge $xy$ and the vertices $v_1,v_s,v_k^-$, by Fact \ref{fact2}, combine with the induced-$(P_2+3P_1)$ freeness of $G$, it follows that 
$v_1\leftrightarrow v_k^-$.
By symmetry,
$v_s\leftrightarrow v_k^+.$

\begin{claim}\label{v_1v_s}
\begin{enumerate}
\item [\rm{(1)}] $	v_1\leftrightarrow w,$ for every  $w\in V(v_2\overrightarrow{P}v_k^-).$ 
\item [\rm{(2)}] $v_s\leftrightarrow w'$, for every $w'\in V(v_k^+\overrightarrow{P}v_{s-1}).$
\end{enumerate}
\end{claim}
\begin{proof}[Proof of Claim \ref{v_1v_s}]
By symmetry, we just prove (1).

Suppose $f=\min\{ i:V(v_i\overrightarrow{P}v_{k-1})\subseteq N_G(v_1)\}$. If $f=2,$ we are done, so assume that $f>2$. Thus, $v_{f-1}\nleftrightarrow v_1$ and $v_f\leftrightarrow v_1$. 
Choose a vertex $x\in V(H)$ such that
$x\leftrightarrow v_k.$ Recall that $R=\{v_k\}$ and hence $x \nleftrightarrow \{v_1,v_k^+,v_{f-1},v_s\}$. 
Considering the edge $v_k^+v_s$ and the vertices $v_{f-1},x,v_1$, by Fact \ref{fact2},
it follows that 
either $v_{f-1}\leftrightarrow v_k^+$ or $v_{f-1}\leftrightarrow v_s$.

Indeed, if $v_{f-1}\leftrightarrow v_k^+$, then
$$
v_s\overleftarrow{P}v_k^+
v_{f-1}\overleftarrow{P}v_1
v_f\overrightarrow{P}v_kx
$$
is a path longer than $P$, a contradiction. If $v_{f-1}\leftrightarrow v_s$, then
$$
xv_k\overrightarrow{P}v_s
v_{f-1}\overleftarrow{P}v_1
v_f\overrightarrow{P}v_{k-1}
$$
is a path longer than $P$, a contradiction. Thus, (1) holds.

This completes the proof of Claim \ref{v_1v_s}.
\end{proof}

By Lemma \ref{pre}(1), $G[V(P)]$ contains no Hamilton cycle. 
By Claim \ref{v_1v_s}, $V(v_1\overrightarrow{P}v_{k-2})$ is anticomplete to $V(v_{k+2}\overrightarrow{P}v_s)$. We next assert that $v_{k-1}$ is anticomplete to $V(v_{k+2}\overrightarrow{P}v_s)$. Otherwise, let $y\in N_G(v_{k-1})\cap V(v_{k+2}\overrightarrow{P}v_s)$. Choose $x\in H$ with $x\leftrightarrow v_k$. We find a longer path 
$$v_1\overrightarrow{P}v_{k-1}y\overrightarrow{P}v_sy^-\overleftarrow{P}v_kx,$$
which contradicts the maximality of $P$. Similarly, $v_{k+1}$ is anticomplete to $V(v_1\overrightarrow{P}v_{k-2})$.
There exists at most one edge between
$V(v_1\overrightarrow{P}v_{k-1})$ and $V(v_{k+1}\overrightarrow{P}v_s)$.  In particular, if it exists, the edge is $v_{k-1}v_{k+1}$.

\begin{claim}\label{clique}
$V(v_1\overrightarrow{P}v_{k-1})$, $V(v_{k+1}\overrightarrow{P}v_s)$ and V(H)
are cliques.
\end{claim}
\begin{proof}[Proof of Claim \ref{clique}]

We first show that both $V(v_1\overrightarrow{P}v_{k-1})$ and $V(v_{k+1}\overrightarrow{P}v_s)$ are cliques.
Suppose to the contrary that there exist two nonadjacent vertices $w_1,w_2\in V(v_1\overrightarrow{P}v_{k-1}).$ Recall that there exists at most one edge $v_{k-1}v_{k+1}$ between
$V(v_1\overrightarrow{P}v_{k-1})$ and $V(v_{k+1}\overrightarrow{P}v_s)$. Now $v_s\nleftrightarrow \{w_1,w_2\}$.
Choose an edge $xy\in E(H)$. Since $w_1,w_2,v_s$ are pairwise nonadjacent to $x$ and $y$, the set
$$
\{x,y\}\cup\{w_1,w_2,v_s\}
$$
induces a copy of $P_2+3P_1$, a contradiction. Thus, $V(v_1\overrightarrow{P}v_{k-1})$ is a clique. Similarly, $V(v_{k+1}\overrightarrow{P}v_s)$ is a clique.

Next we show that $H$ is complete. Suppose not. Then there exist nonadjacent vertices
$x_1,x_2\in V(H).$
Since $v_1v_2\in E(G)$ and $v_s$ is nonadjacent to $H$, the set
$$
\{v_1,v_2\}\cup\{x_1,x_2,v_s\}
$$
induces a copy of $P_2+3P_1$, a contradiction. 

This completes the proof of Claim \ref{clique}.
\end{proof}
Since $deg_G(v_k)\le deg_G(u)$ and $u$ has at most $|V(H)|$ neighbours in $G$, we have
$deg_G(v_k)\le deg_G(u)\le |V(H)|.$
Thus,
$|N_G(v_k)\cap V(H)|\le |V(H)|-2.$
Hence there exist two vertices $x_1,x_2\in V(H)$ such that
$v_k\nleftrightarrow x_1$ and $v_k\nleftrightarrow x_2.$

Set
$$
S_1=V(v_1\overrightarrow{P}v_{k-2}),
\qquad
S_2=V(v_{k+2}\overrightarrow{P}v_s).
$$
By the maximality of $P$, $|S_1|\ge |H|-1$ and $|S_2|\ge |H|-1$.
We claim that either
$deg_{S_1}(v_k)=|S_1|$ or
$deg_{S_2}(v_k)=|S_2|.$
Otherwise, choose
$w_1\in S_1,$ and $w_2\in S_2$
such that
$v_k\nleftrightarrow w_1$ and $v_k\nleftrightarrow w_2.$
Then
$$
\{x_1,x_2\}\cup\{v_k,w_1,w_2\}
$$
induces a copy of $P_2+3P_1$, a contradiction.

Without loss of generality, assume that
$deg_{S_1}(v_k)=|S_1|.$
Then
$$
deg_G(v_k)\ge deg_{S_1}(v_k)+|\{v_{k-1},v_{k+1}\}|+deg_H(v_k)>|V(H)|\ge deg_G(u),
$$
which contradicts $deg_G(u)=\Delta(G)$. This finishes Case 1.

\medskip
\noindent\textbf{Case 2.} $|V(H)|=1$.

Now $V(H)=\{u\}$. Then $R=N_P(u)$ and $deg_G(u)=|R|=\Delta(G)\ge 3$.

We first show that $k=2$ and $\ell=s-1$.
Suppose to the contrary that $k\ge 3$. By Fact \ref{fact2}, $R^+\cup\{v_1\}$ is an independent set.
Consider the edge $v_1v_2$ and the vertex set $R^+\cup\{u\}$, since $G$ is induced-$(P_2+3P_1)$ free, $v_2$ is nonadjacent to at most two vertices of $R^+\cup\{u\}$. Note that $k\ge 3$. Hence, $v_3\notin R^+$. Since $v_2\leftrightarrow \{v_1,v_3\}$, it follows that
$$deg_G(v_2)\ge |R^+\cup\{u\}|-2+|\{v_1,v_3\}|>|R|=deg_G(u),$$
a contradiction. Therefore $k=2$. By symmetry, $\ell=s-1$.
\begin{claim}\label{onlyone}
$G-V(P)=H$. 
\end{claim}
\begin{proof}[Proof of Claim \ref{onlyone}]
To the contrary, suppose $H'$ is another component of $G-P$. If $|H'|\ge 2$, choose an edge $xy\in E(H')$. It is easy to check that $$\{x,y\}\cup \{u,v_1,v_s\}$$ induces a copy of $(P_2+3P_1)$, a contradiction. Thus, $|H'|=1$ and suppose $H'=\{u'\}$.

We assert that $u'$ is nonadjacent to any vertex of $R^+$. Otherwise, suppose $u'\leftrightarrow x$, where $x\in R^+$. By Lemma \ref{pre}(2), $x\ne v_s$.  Then $$v_1\overrightarrow{P}x^-uv_{s-1}\overleftarrow{P}xu'$$ is a longer path than $P$, a contradiction. By Fact~\ref{fact2}, $R^+\cup\{v_1\}$ is an independent set. Consider the edge $v_1v_2$ and the set $(R^+\setminus\{v_3\})\cup \{u'\}$. It follows that $v_2$ is nonadjacent to at most two vertices of $(R^+\setminus\{v_3\})\cup\{u'\}$. Thus,
$$deg_G(v_2)\ge |(R^+\setminus\{v_3\})\cup\{u'\}|-2+|\{v_1,v_3,u\}|=|R|+1>deg_G(u),$$
a contradiction. 
This completes the proof of Claim \ref{onlyone}.
\end{proof}  

Recall that $deg_G(u)=\Delta(G)\ge 3$. Hence, $R\setminus\{v_2,v_{s-1}\}\ne \emptyset.$ 
\begin{claim}\label{v_j}
For any vertex $v_j\in R\setminus\{v_2,v_{s-1}\}$, $v_j^+\nleftrightarrow v_j^-$. 
\end{claim}
\begin{proof}[Proof of Claim \ref{v_j}]
To the contrary, suppose $v_j^+\leftrightarrow v_j^-$. 
Consider the edge $v_j^+v_j^-$ and the vertices $v_1,v_s,u$. By Fact \ref{fact2}, $v_1\nleftrightarrow v_j^+$ and $v_s\nleftrightarrow v_j^-$. Since $G$ is induce-$(P_2+3P_1)$ free, either $v_1\leftrightarrow v_j^-$ or $v_s\leftrightarrow v_j^+$.

If $v_1\leftrightarrow v_j^-$, then $$v_1v_j^-\overleftarrow{P}v_2uv_j\overrightarrow{P}v_s$$ is a Hamilton path, a contradiction. Thus, $v_s\leftrightarrow v_j^+$. Then  $$v_sv_j^+\overrightarrow{P}v_{s-1}uv_j\overleftarrow{P} v_1$$
is a Hamilton path,
a contradiction.
\end{proof}

Now we distinguish two cases according to $\Delta(G)$.

{ \bf Subcase 2.1} $\Delta(G)=3$.

Since $deg_G(u)=3$, we may write $R=\{v_2,v_j,v_{s-1}\}$ where $v_j\in R\setminus\{v_2,v_{s-1}\}$.
Recall that $v_1\nleftrightarrow \{v_j^+,v_s\}$ and $v_s\nleftrightarrow \{v_j^-,v_s\}$. By Claim \ref{v_j}, $v_j^-\nleftrightarrow v_j^+$. We assert that 
$\{v_1,v_j^+,v_j^-,v_s\}$ is an independent set. It suffices to show that $v_1\nleftrightarrow v_j^-$ and $v_s\nleftrightarrow v_j^+$. If $v_1\leftrightarrow v_j^-$, then $$v_1v_{j-1}\overleftarrow{P}v_2uv_j\overrightarrow{P}v_s$$ is a Hamilton path of $G$, a contradiction. Thus, $v_1\nleftrightarrow v_j^-$.  Similarly, $v_s\nleftrightarrow v_j^+$. Hence, $\{v_1,v_j^+,v_j^-,v_s\}$ is an independent set. 
Since $deg_G(v_2)=3$, if $v_j^-\ne v_3$, then $$\{v_1,v_2\}\cup\{v_j^+,v_j^-,v_s\}$$ induces a copy of $P_2+3P_1$, a contradiction. Hence $v_j^-=v_3$.
Similarly, since $deg_G(v_{s-1})=3$, we get $v_j^+=v_{s-2}$.
Therefore $s=7$ and $n=8$. Therefore, $G[V(P)]=P_7$.

Set $X=\{v_2,v_4,v_6\},$ $Y_0=\{v_1,v_3,v_5,v_7\}$ and  $Y=Y_0\cup\{u\}.$
Let $F=v_5v_2v_7v_4v_1v_6v_3,$
and hence $F=P_7$. Therefore, $G=K_{3,5}-E(F)\in\mathcal G_3.$

{\bf Subcase 2.2} $\Delta(G)\ge 4$.

By Fact~\ref{fact2}, the set $R^+\cup\{v_1\}$ is independent. Consider the edge $v_1v_2$ and the independent set $R^+\setminus\{v_3\}$.
Since $G$ is induced-$(P_2+3P_1)$ free,  $v_2$ is nonadjacent to at most two vertices of $R^+\setminus\{v_3\}$. Hence
$$deg_G(v_2)\ge |R^+\setminus\{v_3\}|-2+|\{v_1,v_3,u\}|=|R|=deg_G(u).$$
Since $deg_G(u)=\Delta(G)$, the above inequality becomes equality. Hence, $v_2$ is exactly nonadjacent to two vertices of $R^+\setminus\{v_3\}$.  Therefore,
$$|R^+\setminus N_G(v_2)|=2\qquad \text{and}\qquad N_G(v_2)\subseteq R^+\cup\{v_1,u\}.$$
By symmetry, $deg_G(v_{s-1})=deg_G(u)=\Delta(G)$ and $N_G(v_{s-1})\subseteq R^-\cup\{v_s,u\}$ .

Now let $v_j\in R\setminus\{v_2,v_{s-1}\}$.
Consider the edge $v_jv_{j+1}$ and the independent set $(R^+\setminus\{v_{j+1}\})\cup\{v_1\}$.
Since G is induced-$(P_2+3P_1)$ free, $v_j$ is nonadjacent to at most two vertices of this set. Thus
$$
\Delta(G)
\ge deg_G(v_j)
\ge |\{v_{j+1},u\}|
+\bigl|(R^+\setminus\{v_{j+1}\})\cup\{v_1\}\bigr|-2
\ge |R|
=deg_G(u).
$$
Inequalities becomes equalities, and  hence 
$v_j$ is exactly nonadjacent to two vertices of $(R^+\setminus\{v_{j+1}\})\cup\{v_1\}$.
We obtain that
\begin{equation}\label{eqv_j}
N_G(v_j)\subseteq R^+\cup\{v_1,u\}.
\end{equation}
Thus, $v_{j-1}\in R^+$ and hence $v_{j-2}\in R$.
Since $v_j$ was chosen arbitrarily, we get $R^+\setminus\{v_s\}=R^-\setminus\{v_1\}$.
Let $t=|R|=\deg_G(u)=\Delta(G)$. 

Set $X=R$, $Y_0=R^+\cup\{v_1\}$, and $Y=Y_0\cup\{u\}$. Then
$|X|=t$, $|Y_0|=t+1$, and $|Y|=t+2$. By
Claim~\ref{onlyone}, we have $V(G)=X\cup Y$.

Relabel the vertices of $X$ and $Y_0$ according to their order
along $P$ so that
$$
X=\{x_1,\ldots,x_t\},\qquad
Y_0=\{y_1,\ldots,y_{t+1}\},
$$
and
$$
P=y_1x_1y_2x_2\cdots y_tx_ty_{t+1}.
$$
In particular, $y_ix_i,x_iy_{i+1}\in E(G)$ for every
$1\le i\le t$.

By Fact~\ref{fact2}, $Y_0$ is independent. Moreover,
$N_G(u)=R=X$, so $u$ is anticomplete to $Y_0$. Hence $Y$ is
independent. Together
with \eqref{eqv_j}, it follows that $N_G(x)\subseteq Y$ for every
$x\in X$. Thus, $X$ is also independent, and consequently $G$ is
bipartite with bipartition $(X,Y)$. Notice also that $u$ is complete
to $X$.

Let $K=K_{t,t+2}$ be the complete bipartite graph with bipartition
$(X,Y)$. Define a spanning subgraph $F$ of $K[X,Y_0]$ with $E(F)=E\bigl(K[X,Y_0]\bigr)\setminus E(G).$
Thus, for $x\in X$ and $y\in Y_0$, we have $xy\in E(F)$ if and
only if $xy\notin E(G)$. Since $u$ is complete to $X$, it follows
that $G=K-E(F).$

We next determine the structure of $F$. Let $x\in X$. Consider the
edge $ux$ and the independent set $Y_0$. Since $G$ is
induced-$(P_2+3P_1)$ free, $x$ is nonadjacent to at most two
vertices of $Y_0$. Hence $\deg_F(x)\le2$. On the other hand, $\deg_G(x)=1+|Y_0|-\deg_F(x)=t+2-\deg_F(x).$
Since $\deg_G(x)\le\Delta(G)=t$, we have $\deg_F(x)\ge2$.
Therefore,
$$
\deg_F(x)=2,\qquad\text{for every }x\in X.
$$

Now let $y\in Y_0$. Since $y$ lies on $P$, it has a neighbour
$x\in X$ on $P$. If $y$ were nonadjacent to three vertices of
$X$, then these three vertices, together with the edge $xy$, would
induce a copy of $P_2+3P_1$. Therefore,
$\deg_F(y)\le2$ for every $y\in Y_0$.

Counting the edges of $F$ from its two parts, we obtain
$$
\sum_{y\in Y_0}\deg_F(y)
=\sum_{x\in X}\deg_F(x)=2t.
$$
Since $|Y_0|=t+1$, it follows that
$$
\sum_{y\in Y_0}\bigl(2-\deg_F(y)\bigr)=2.
$$
As $0\le\deg_F(y)\le2$ for every $y\in Y_0$, exactly one of the
following two cases occurs.

First, there are exactly two vertices $y',y''\in Y_0$ of degree
one in $F$, and every other vertex of $Y_0$ has degree two. Since
every vertex of $X$ has degree two, the component of $F$ containing
$y'$ is a path with endvertices $y'$ and $y''$, while every other
component is a cycle.

Second, there is exactly one vertex $y'\in Y_0$ of degree zero in
$F$, and every other vertex of $Y_0$ has degree two. In this case,
$y'$ is an isolated vertex of $F$, while every other component of
$F$ is a cycle.

Since $F$ is bipartite, all its cycles are even. Thus, in either
case, $F=Q+C_1+\cdots+C_r,$
where $C_1,\ldots,C_r$ are pairwise vertex-disjoint even cycles and
$Q$ is a path whose endvertices belong to $Y_0$. In the second
case, $Q$ is the trivial path consisting of the single vertex $y'$.
Moreover, these components together contain all vertices of
$X\cup Y_0$.

Therefore $F\in\mathcal F_t$. Since $G=K_{t,t+2}-E(F)$, we conclude
that $G\in\mathcal G_t\subseteq\mathcal G.$

This completes the proof of Proposition~\ref{P_2+3P_1}.
\end{proof}
We next construct counterexamples for the linear forests
not covered by Propositions~\ref{P_3+2P_1} and
\ref{P_2+3P_1}. More precisely, for each of the remaining linear
forests $F$, we construct a connected induced-$F$ free graph
containing a top vertex that is not a Gallai vertex.

We now present four constructions about the remaining linear forests of order five. Construction~A provides
counterexamples for $P_3+P_2$, Construction~B deals with
$2P_2+P_1$, Construction~C  deals with $P_4+P_1$ and
$P_5$, and Construction~D provides counterexamples for
$5P_1$.

\medskip
\noindent
\textbf{Construction A.}
Let $t\ge8$. We first construct a base graph $G_0(t,L)$, where $L\subseteq \{0,1,\ldots,t\}$, $|L|\ge5$, and no two integers in $L$ are consecutive.
Let $X:=\{u_{2j}:1\le j\le t\}$ and $Y:=\{u_{2i+1}:0\le i\le t\}$. Define $Y_1:=\{u_{2i+1}:i\in L\}$ and $Y_2:=Y\setminus Y_1$.
The vertex set of $G_0(t,L)$ is $V(G_0)=X\cup Y$.
The edge set is defined as follows.
\begin{enumerate}
\item[\rm(a)] Add all edges between $Y_2$ and $X$.
\item[\rm(b)] For $u_{2i+1}\in Y_1$, add the following edges:
\begin{itemize}
\item[(b1)] if $1\le i\le t-1$, add
$u_{2i}u_{2i+1}$, $u_{2i+1}u_{2i+2}$, and $u_{2i}u_{2i+2}$;
\item[(b2)] if $i=0$, add the edge $u_1u_2$;
\item[(b3)] if $i=t$, add the edge $u_{2t+1}u_{2t}$.
\end{itemize}
\end{enumerate}

\begin{itemize}
\item Even order, $|V(G_1(t,L))|=2t+2$.
\end{itemize}
The graph $G_1(t,L)$ is obtained from $G_0(t,L)$ by adding a new
vertex $u$ and joining $u$ to every vertex of $X$. Thus,
$N_{G_1(t,L)}(u)=X$
and $deg_{G_1(t,L)}(u)=t.$
\begin{itemize}
\item Odd order, $|V(G_2(t,L))|=2t+3$.
\end{itemize}
The graph $G_2(t,L)$ is obtained from $G_1(t,L)$ by adding a new vertex $u'$ and joining $u'$ to every vertex of $X$.
Thus, $N_{G_2(t,L)}(u)=N_{G_2(t,L)}(u')=X$
and $deg_{G_2(t,L)}(u)=deg_{G_2(t,L)}(u')=t.$

\begin{figure}[h]
\centering
\includegraphics[width=1.0\textwidth]{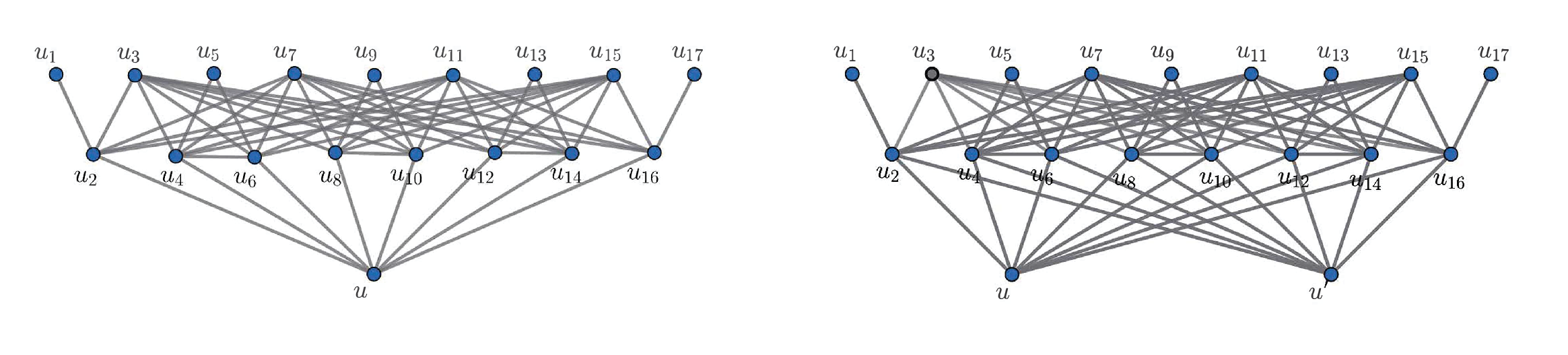}
\caption{The graphs $G_1(t,L)$ and $G_2(t,L)$ ($t=8$, $L=\{0,2,4,6,8\}$).}
\end{figure}

The required properties of Construction~A are summarized in the
following claim.
\begin{claim}\label{consA}
The graphs $G_1(t,L)$ and $G_2(t,L)$ are connected and
induced-$(P_3+P_2)$ free. Moreover, in each graph, the vertex $u$ is
a top vertex but not a Gallai vertex.
\end{claim}

\begin{proof}[Proof of Claim \ref{consA}]
By definition, $G_1(t,L)$ and $G_2(t,L)$ are connected. Note that, for $i=1,2$,
$\Delta(G_i(t,L))=t=deg_{G_i(t,L)}(u)$, and hence $u$ is a top vertex of $G_i(t,L)$.

We assert that $G_i(t,L)$ is nontraceable, for $i=1,2$. Let $\widetilde{G_i}(t,L)$ be a graph obtained from $G_i(t,L)$ by adding a new vertex $v$ joining all vertices to $G_i(t,L)$. Clearly, $\widetilde{G_i}(t,L)-X\cup\{v\}$  contains at least $|X|+2$ components. This implies that $\widetilde{G_i}(t,L)$ is not 1-tough. Hence, $\widetilde{G_i}(t,L)$ is nonhamiltonian and hence $G_i(t,L)$ is nontraceable.

The set $Y\cup\{u\}$ is independent and has size $t+2$, whereas
$|X|=t$. Thus, every path of $G_1(t,L)$ has order at most $2t+1$.
Since $P=u_1u_2\cdots u_{2t+1}$
is a longest path of order $2t+1$ avoiding $u$ in $G_1(t,L)$, $u$ is not a Gallai vertex of  $G_1(t,L)$. 	
The set $Y\cup\{u,u'\}$ is independent and has size $t+3$, whereas
$|X|=t$. Thus, every path of $G_2(t,L)$ has order at most $2t+1$. Since $P'=u_1u_2\cdots u_{2t+1}$
is a longest path of order $2t+1$ avoiding $u$ in $G_2(t,L)$, $u$ is not a Gallai vertex of $G_2(t,L)$.

We next show that $G_1(t,L)$ is induced-$(P_3+P_2)$ free. Let $Q$ be
an induced path of order three. It is enough to show that the vertices
anticomplete to $Q$ form an independent set. No vertex of $Y_1$ can
be the middle vertex of $Q$, since the two neighbours of such a vertex,
when both exist, are adjacent. If the middle vertex belongs to $Y_2$,
then it is complete to $X$, so every vertex anticomplete to $Q$
belongs to the independent set $Y_1\cup\{u\}$. If the middle vertex
is $u$, then every vertex anticomplete to $Q$ belongs to the
independent set $Y$. Finally, consider that the middle vertex belongs to $X$. If one endpoint of $Q$ is $u$, then the vertices anticomplete to
$Q$ belong to $Y$ and hence are independent. If one endpoint of $Q$ belongs to $X$, then the vertices anticomplete to
$Q$ belong to $Y_1$. Since no two
elements of $L$ are consecutive, the vertices anticomplete to
$Q$ are again independent. If both endpoints of $Q$ belong to $Y$, then the vertices anticomplete to
$Q$ belong to $Y$ and hence are independent. 
Therefore, no edge is anticomplete
to an induced $P_3$, and hence $G_1(t,L)$ is induced-$(P_3+P_2)$ free.

Now we show that $G_2(t,L)$ is induced-$(P_3+P_2)$ free.
By the symmetry of $u$ and $u'$, if  $G_2(t,L)$  contains a copy of  induced-$(P_3+P_2)$, then it contains both $u$ and $u'$. Since $u$ and $u'$ are nonadjacent, $u$ and $u'$ must be the two endpoints of the $P_3$, that is, the two vertices forming
the $P_2$ must belong to $Y$, which contradicts that $Y$ is
independent. Thus, $G_2(t,L)$ is induced-$(P_3+P_2)$ free.
\end{proof}

\medskip
\noindent
\textbf{Construction B.}
\begin{itemize}
\item Even order, $|V(G_3)|=2t+2$.
\end{itemize}
Fix integers $t\ge 4$ and $2\le s\le t-2$. Let
$X_1=\{a_1,a_2,\ldots,a_s\}$, $X_2=\{b_1,b_2,\ldots,b_s\}$,
$Y_1=\{a_{s+1},a_{s+2},\ldots,a_t\}$, and
$Y_2=\{b_{s+1},b_{s+2},\ldots,b_t\}$.
Set $V(G_3)=X_1\cup X_2\cup Y_1\cup Y_2\cup\{u,w\}$.
The edge set $E(G_3)=E_1\cup E_2\cup E_3\cup E_4$ is given by
$$
E_1=\{a_ib_j:a_i\in X_1,\ b_j\in X_2\},\qquad
E_2=\{a_ib_j:a_i\in Y_1,\ b_j\in Y_2\},
$$
$$
E_3=\{a_iw:a_i\in X_1\cup Y_1\},\qquad
E_4=\{a_iu:a_i\in X_1\cup Y_1\}.
$$
Note that $\{u,w\}$ is complete to $X_1\cup Y_1$ , while $X_1$ is anticomplete to $Y_2$ and $Y_1$ is anticomplete to $X_2$.

\begin{itemize}
\item Odd order, $|V(G_4)|=2t+3$.
\end{itemize}
With the same notation as in Construction~B, take $V(G_4)=X_1\cup X_2\cup Y_1\cup Y_2\cup\{u,w,w'\}$.
Define
$$
E_1=\{a_ib_j:a_i\in X_1,\ b_j\in X_2\},\qquad
E_2=\{a_ib_j:a_i\in Y_1,\ b_j\in Y_2\},
$$
$$
E_3=\{a_iw:a_i\in X_1\},\qquad
E_4=\{a_iw':a_i\in Y_1\},\qquad
E_5=\{a_iu:a_i\in X_1\cup Y_1\},
$$
and set $E(G_4)=E_1\cup E_2\cup E_3\cup E_4\cup E_5\cup\{ww'\}$.

\begin{figure}[h]
\centering
\includegraphics[width=0.7\textwidth]{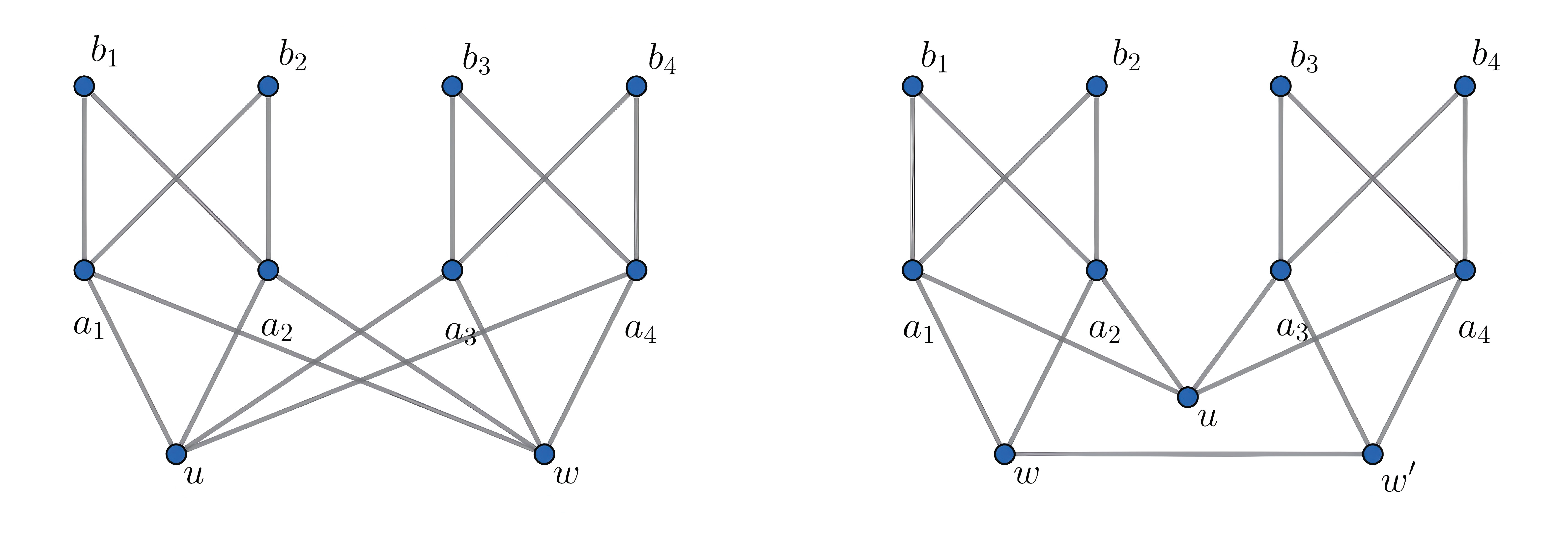}
\caption{The graphs $G_3$ and $G_4$ ($t=4$, $s=2$)}
\end{figure}

The properties of Construction~B are below.

\begin{claim}\label{consB}
The graphs $G_3$ and $G_4$ are connected and
induced-$(2P_2+P_1)$ free. Moreover, in each of these graphs, the
vertex $u$ is a top vertex but not a Gallai vertex.
\end{claim}

\begin{proof}[Proof of Claim~\ref{consB}]
By construction, $G_3$ and $G_4$ are connected. Note that for $i=3,4$, $\Delta(G_i)=t=deg_{G_i}(u)$, and
hence $u$ is a top vertex of $G_i$. 

We assert that $G_i$ is nontraceable, for $i=3,4$. Let $\widetilde{G_i}$ be a graph obtained from $G_i$ by adding a new vertex $v$ joining all vertices to $G_i$. Clearly, $\widetilde{G_i}-X_1\cup Y_1\cup\{v\}$  contains $t+2$ components. This implies that $\widetilde{G_i}$ is not 1-tough. Thus, $\widetilde{G_i}$ is nonhamiltonian and hence $G_i$ is nontraceable.

Since $Q=b_{s+1}a_{s+1}b_{s+2}a_{s+2}\cdots b_ta_tw
a_1b_1\cdots a_sb_s$
is a longest path of order $2t+1$ avoiding $u$ in $G_3$, $u$ is not a Gallai vertex of $G_3$. Since $Q'=b_{s+1}a_{s+1}b_{s+2}a_{s+2}\cdots b_ta_tw'wa_1b_1\cdots a_sb_s$
is a longest path of order $2t+2$  avoiding $u$ in $G_4$, $u$ is not a
Gallai vertex of $G_4$.

We next verify that $G_3$ is induced-$(2P_2+P_1)$ free. If four
vertices induce a copy of $2P_2$, then one edge must join $X_1$ to
$X_2$ and the other must join $Y_1$ to $Y_2$. Indeed, two edges
within the same complete bipartite subgraph have an additional edge
between their endpoints, and an edge incident with $u$ or $w$ cannot
form an induced $2P_2$ with another edge. Every remaining vertex is
adjacent to at least one endpoint of these two edges. Hence no vertex
can be isolated from an induced $2P_2$, and $G_3$ is
induced-$(2P_2+P_1)$ free.

Finally, we verify that $G_4$ is induced-$(2P_2+P_1)$ free. The edge
$ww'$ cannot belong to an induced $2P_2$, since every other edge has
an endpoint adjacent to $w$ or $w'$. A direct observation shows that
every induced $2P_2$ not containing $ww'$ is of one of the following
types: one edge between $X_1$ and $X_2$ together with one edge between
$Y_1$ and $Y_2$; one edge between $X_1$ and $X_2$ together with an
edge joining $w'$ to $Y_1$; or one edge joining $w$ to $X_1$ together
with an edge between $Y_1$ and $Y_2$. In each case, every remaining
vertex is adjacent to at least one endpoint of the induced $2P_2$.
Therefore, no isolated vertex can be added to it, and $G_4$ is
induced-$(2P_2+P_1)$ free.
\end{proof}

\medskip
\noindent\textbf{Construction C.}
For every integer $n\ge 10$, define a graph $G_5$ of order $n$ as follows. Let
$V(G_5)=\{u,v,a_1,a_2,b_1,b_2\}\cup C$, where $C=\{c_1,c_2,\ldots,c_{n-6}\}$.
The edge set is
$$
E(G_5)=\{a_1a_2,b_1b_2,uv\}\cup\{va_i: i=1,2\}\cup\{vb_i: i=1,2\}\cup\{uc_j:1\le j\le n-6\}.
$$

\begin{figure}[h]
\centering
\includegraphics[width=0.45\textwidth]{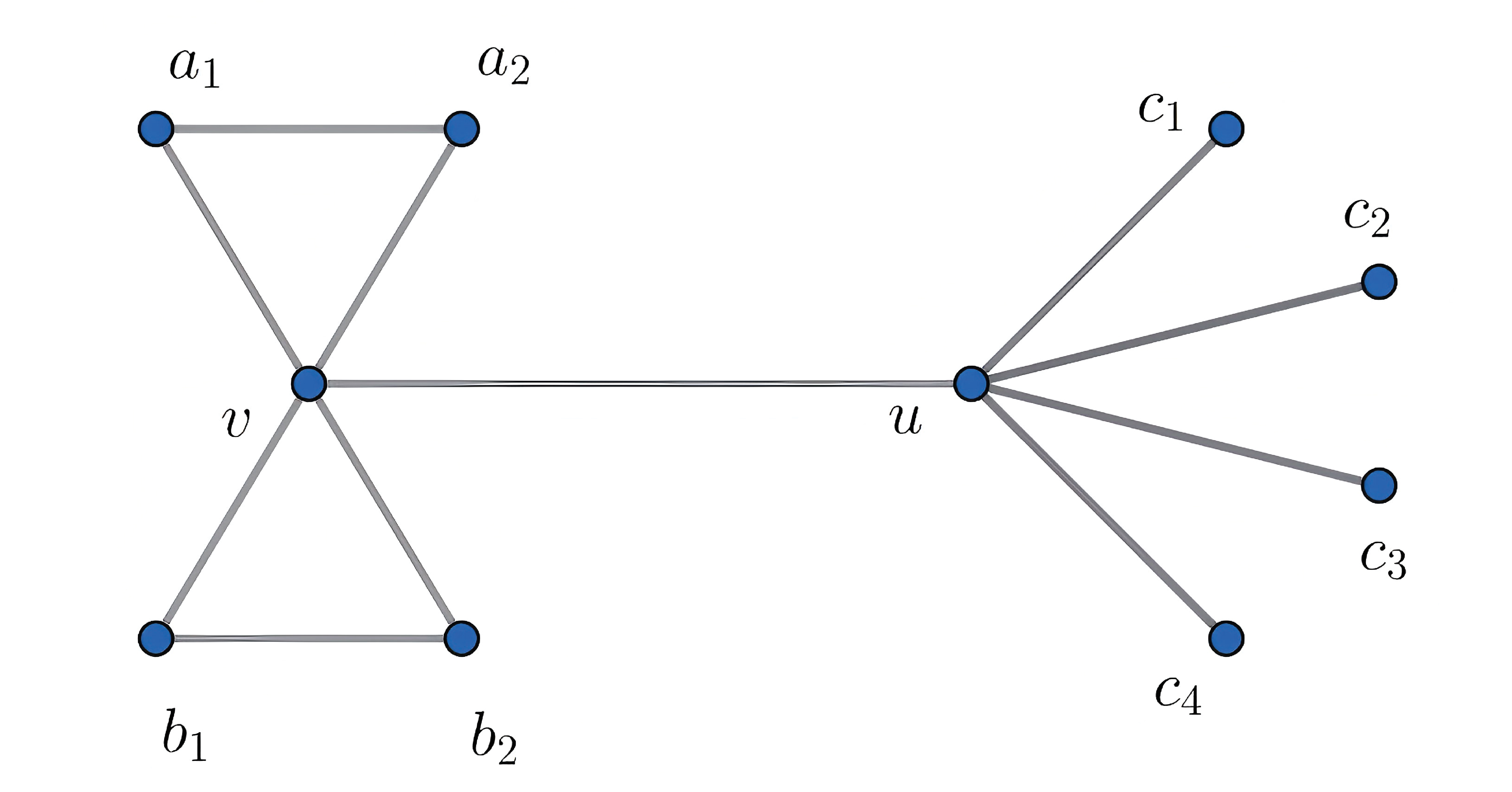}
\caption{The graph $G_5$ of order $n=10$}
\end{figure}

The properties of Construction~C are below.

\begin{claim}\label{consC}
For every integer $n\ge10$, the graph $G_5$ is a connected,
induced-$(P_4+P_1)$ free, induced-$P_5$ free graph of order $n$ containing a top vertex $u$
that is not a Gallai vertex.
\end{claim}

\begin{proof}[Proof of Claim~\ref{consC}]
By construction, $G_5$ is connected.  Since $n\ge10$, we have
$\Delta(G_5)=n-5=deg_G(u)$, and hence $u$ is a top vertex. Note that 
$$P=a_1a_2vb_1b_2$$
is a longest path of order five that avoids $u$ in $G_5$. 
Thus $u$ is not a Gallai vertex. 	 

It remains to verify the two forbidden induced subgraph conditions.
Since $G_5-v$ contains no induced $P_4$, every induced $P_4$ in
$G_5$ must contain $v$. Note that $v$ lies in two triangles $va_1a_2v$ and $vb_1b_2v$. The two triangles containing $v$ imply that
every induced $P_4$ contains vertices $u$ and $c_j$, where $c_j\in C$. Every vertex outside this path has a
neighbour on it. Indeed, each vertex of $C\setminus\{c_j\}$ is adjacent
to $u$, the other vertex in the triangle is adjacent
to $v$. Hence no vertex is anticomplete to an induced $P_4$, and
$G_5$ is induced-$(P_4+P_1)$ free.

A similar analysis also shows that no induced $P_4$ can be extended
to an induced $P_5$. Therefore, $G_5$ is induced-$P_5$
free.
\end{proof}

\medskip
\noindent
\textbf{Construction D.}
Long, Milans, and Munaro \cite{LMM} gave an example that shows that a top vertex of a connected induced-$5P_1$ free is not a Gallai vertex. However, their  example has order $3p+1$, where $p\ge 3$. For the completeness of the proof, we construct a new graph $G_6$ of order $n\ge 18$ as follows:

Let $s$ and $t$ be integers with $3\le s\le t$. Let $H_1$ be a complete graph of order $s$ and $H_i$ be a complete graph of order $t$, for each $i\in \{2,3,4\}$. Choose two vertices $v,w$ in $H_1$ and choose a vertex $u_i$ in each $H_i$. By adding four edges $u_2v$, $u_3w$, $u_4v$ and $u_4w$, we obtain a graph $G_6$. Hence, $|G_6|=n=3t+s$, where  $3\le s\le t$. 
\begin{figure}[h]
\centering
\includegraphics[width=0.5\textwidth]{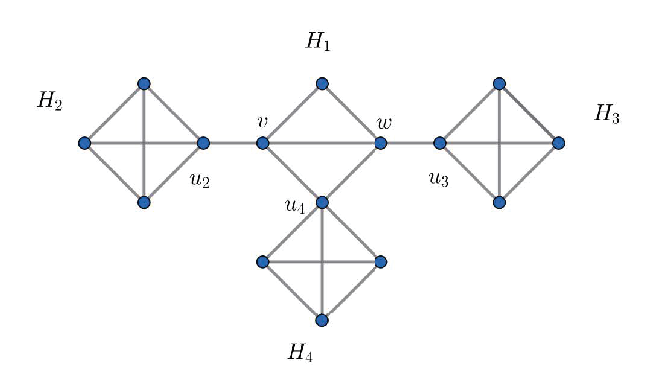}
\caption{The graph $G_6$ ($s=3$, $t=4$)}
\end{figure}

The properties of Construction~D are below.

\begin{claim}\label{consD}
Let $s$ and $t$ be  integers with $3\le s\le t$. $G_6$ is a connected
induced-$5P_1$ free graph of order $n$ containing a top vertex 
that is not a Gallai vertex.
\end{claim}
\begin{proof}[Proof of Claim~\ref{consD}]
By the definition of $G_6$, $G_6$ is connected. Since $3\le s\le t$, it follows that $\Delta(G_6)=deg_{G_6}(u_4)=t+1$, that is, $u_4$ is a top vertex of $G_6$. Since for each $j\in\{1,2,3,4\}$, each $V(H_j)$ is a clique,  every path of $G_6$ has order at most $2t+s.$ In addition, there exists a longest path $P$ containing all vertices of $H_1$, $H_2$ and $H_3$, that is, $P$ is a longest path avoiding vertex $u_4$. Hence, $u_4$ is a top vertex but not a Gallai vertex. 

By the definition of $G_6$,  it follows that $G_6$ has independence number four. This implies that $G_6$ is induced-$5P_1$ free.
\end{proof}

The preceding results cover all linear forests of order five. Next we record a simple observation concerning
linear forests of order at least six.

\begin{lemma}\label{remains}
Every linear forest of order at least six contains, as an induced
subgraph, at least one of the following linear forests:
$$
P_5,\quad
P_4+P_1,\quad
P_3+P_2,\quad
2P_2+P_1,\quad
P_2+3P_1,\quad
5P_1.
$$
\end{lemma}

We are now ready to prove the main theorem.

\begin{proof}[\bf Proof of Theorem~\ref{main}]
We prove the two directions separately.

Suppose first that $H$ is a linear forest of order at most four
or
$H=P_3+2P_1.$
If $H$ is a linear forest of order at most four, the conclusion follows directly from
Theorem  \ref{lmm2}. 
If $H=P_3+2P_1,$
the conclusion follows from Proposition~\ref{P_3+2P_1}. This proves
the sufficient condition.

Conversely, we prove the necessary condition. By Theorem \ref{lmm1}, $H$ must be a linear forest of order at most nine. If $|H|\le 4$, we are done, so assume that $|H|\ge5$
and $H\ne P_3+2P_1.$ 
We show that there exists a connected induced-$H$ free graph
containing a top vertex that is not a Gallai vertex.

We first consider the case
$|H|=5.$ 
Since $H$ is a linear forest of order five and
$H\ne P_3+2P_1$, it is isomorphic to one of	$\{P_5, P_4+P_1, P_3+P_2, 2P_2+P_1, P_2+3P_1, 5P_1\}.$

For $H=P_2+3P_1$, Proposition~\ref{P_2+3P_1} shows that every graph
in $\mathcal{G}$ is a connected induced-$(P_2+3P_1)$ free graph
containing a top vertex that is not a Gallai vertex.

For the remaining five linear forests, the constructions  provide the required counterexamples. Therefore, the desired
property fails for every linear forest of order five other than
$P_3+2P_1$.

It remains to consider the case
$6\le|H|\le9.$
By Lemma~\ref{remains}, the graph $H$ contains,
as an induced subgraph, a linear forest $F$ such that $F\in
\{P_5,P_4+P_1,P_3+P_2,2P_2+P_1,P_2+3P_1,5P_1\}.$
By the preceding arguments, there exists a connected induced-$F$
free graph $G$ containing a top vertex that is not a Gallai vertex.
This completes the proof of Theorem~\ref{main}. 
\end{proof}

\section*{\normalsize Acknowledgement}  

The author thanks Professor Xingzhi Zhan for his constant support and guidance.

\end{document}